\documentclass[12pt,a4paper,reqno]{amsart}
\usepackage{amssymb,amsmath}
\usepackage{amsfonts,amsbsy,bm}
\usepackage[dvipsnames]{xcolor}
\usepackage{latexsym}
\usepackage{exscale}
\usepackage{enumerate}
\usepackage{soul}
\usepackage{amsthm}

\usepackage{mathtools}

\DeclarePairedDelimiter\floor{\lfloor}{\rfloor}

\usepackage[colorlinks, bookmarks=true]{hyperref}
\usepackage{hyperref}
\usepackage{cases}
\usepackage{color}

\usepackage[utf8]{inputenc}

\newcommand{\N}{\mathbb N}
\newcommand{\B}{\mathcal B}

\newcommand{\be}{{\mathbf e}}

\newcommand {\X} {{\mathbb X}}
\newcommand {\Y} {{\mathbb Y}}

\renewcommand{\phi}{{\varphi}}

\def\supp{\mathop{\rm supp}}
\def\sgn{\mathop{\rm sign}}

\numberwithin{equation}{section}
\newtheorem*{theorem*}{Theorem}
\newtheorem{theorem}{Theorem}[section]
\newtheorem{lemma}[theorem]{Lemma}
\newtheorem{defi}[theorem]{Definition}
\newtheorem{corollary}[theorem]{Corollary}
\newtheorem{Remark}[theorem]{Remark}
\newtheorem{remark}[theorem]{Remark}

\newtheorem{proposition}[theorem]{Proposition}

\newtheorem{question}{Question}
\newtheorem{example}[theorem]{Example}

\newcommand{\Ba}[1]{\begin{array}{#1}}
\newcommand{\Ea}{\end{array}}
\newcommand{\Be}{\begin{equation}}
\newcommand{\Ee}{\end{equation}}
\newcommand{\Bea}{\begin{eqnarray}}
\newcommand{\Eea}{\end{eqnarray}}
\newcommand{\Beas}{\begin{eqnarray*}}
\newcommand{\Eeas}{\end{eqnarray*}}
\newcommand{\Benu}{\begin{enumerate}}
\newcommand{\Eenu}{\end{enumerate}}
\newcommand{\Bi}{\begin{itemize}}
\newcommand{\Ei}{\end{itemize}}

\newcommand{\BR}{\begin{Remark} \em}
\newcommand{\ER}{\end{Remark}}
\newcommand{\BE}{\begin{example} \em}
\newcommand{\EE}{\end{example}}

\newcounter{reg}
\newcounter{regTO}
\newcommand{\bff}{\mathbf 1}
\newcommand{\cupdot}{\mathbin{\mathaccent\cdot\cup}}

\newcommand{\n}{\mathbf{n}}
\newcommand\G{\mathbf{G}}
\newcommand{\K}{\mathbf{K}}

\newcommand{\C}{\mathbf{C}}
\title[Quasi-greedy bases for sequences with gaps]{Quasi-greedy bases for sequences with gaps}

\author[M. Berasategui]{M. Berasategui}

\address{Miguel Berasategui
	\\
	IMAS - UBA - CONICET - Pab I, Facultad de Ciencias Exactas y Naturales \\ Universidad de Buenos Aires \\ (1428), Buenos Aires, Argentina}
\email{mberasategui@dm.uba.ar}

\author[P.\ M. Bern\'a]{P. M. Bern\'a}
\address{Pablo M. Bern\'a\\
	Departamento de Matem\'atica Aplicada y Estad\'istica, Facultad de Ciencias Econ\'omicas y Empresariales, Universidad San Pablo-CEU, CEU Universities\\ Madrid, 28003 Spain.}
\email{pablo.bernalarrosa@ceu.es}

\begin{document}

\thanks{The first author was supported in part by ANPCyT PICT-2018-04104. The second author was partially supported by the grants MTM-2016-76566-P (MINECO, Spain) and 20906/PI/18 from Fundaci\'on S\'eneca (Regi\'on de Murcia, Spain). 
}
\maketitle
\begin{abstract}
	In this paper, we establish new advances in the theory started by T. Oikhberg in \cite{O2015} where the author joins greedy approximation theory with the use of sequences with gaps. Concretely, we address and partially answer three open questions related to quasi-greedy bases for sequences with gaps posed in \cite[Section 6]{O2015}.
\end{abstract}

\section{Background and organization of the paper}
Since 1999, one of the algorithms that has been studied by different researchers in the area of Non-linear Approximation Theory is the Thresholding Greedy Algorithm (TGA). Basically, if $\mathcal B=(\be_i)_{i\in \N}$ is a basis in a Banach space $\X$ and $x=\sum_{i=1}^{\infty}\be_i^*(x)\be_i$, the algorithm selects the biggest coefficients of $x$ in modulus. The pioneering paper in this topic is \cite{KT}, where S. V. Konyagin and V. N. Temlyakov introduced this algorithm. After that, mathematicians such as N. J. Kalton, S. J. Dilworth, D. Kutzarova, P. Wojtaszczyk, among others,  have studied different convergences of the TGA introducing the so called greedy-type bases.

On the other hand, one important topic in Mathematical Analysis is the study of lacunary sequences (or more generally, sequences with gaps). One  important application of these sequences with gaps is in the study of trigonometric series. For instance, a famous result of A. Zygmund is the following: consider a sequence of positive integers $(n_k)_{k=1}^\infty$ with the Hadamard condition, that is, 
$$n_{k+1}>n_k(1+c),\; c>0.$$
Assume that $\sum_{k=1}^\infty a_k^2$ is a divergent series where $a_k's$ are non-negative real numbers. Then, for any sequence of real numbers $(b_k)_{k=1}^\infty$, the trigonometric series
$$\sum_{k=1}^{\infty}a_k \cos(n_k x + b_k)$$
diverges almost everywhere and also is not a Fourier series (information on the lacunary trigonometric series can be found in \cite{Z}). From the point of view of approximation, for example, S. J. Dilworth and  D. Kutzarova proved in \cite{DK} that there exists a strictly increasing sequence of integers $(n_k)_{k=1}^\infty$ such that the lacunary Haar system $\left((h_j^{n_k})_{j=1}^{2^{n_k}}\right)_{k=0}^\infty$ in $H_1$ is an $f(n)$-approximate $\ell_1$ system (see \cite[Proposition 5.3]{DK} for more details).

Recently, T. Oikhberg in \cite{O2015} linked the world of the TGA with that of sequences with gaps. Concretely, the author defined $\n$-quasi-greedy bases for a strictly increasing sequence of positive integers $\n=(n_k)_{k\in \N}$ and he studied some relations between this notion and the usual notion of quasi-greedy bases introduced in \cite{KT}. As an application of his results, he proved that if the trigonometric basis in $L_p$ ($1<p<\infty$) is $\n$-quasi-greedy, then $p=2$ and, for the Haar system, he proved that in the space $L_1(0,1)$, this basis is not $\n$-quasi-greedy. In a independent way, the idea of using strictly increasing sequences $\n$ with the greedy algorithm has been used in several papers to show the behavior of the so-called Lebesgue-type inequalities for the greedy algorithm (see for instance \cite[Section 5.5]{BBG} and \cite[Section 8.5]{BBGHO}). 

In this paper we address three questions posed in \cite[Section 6]{O2015}. The paper is structured as follows: in Section \ref{intro} we give the main definitions of bases in Banach spaces and introduce the TGA, the WTGA and $\n$-$t$-quasi-greedy bases. In Section \ref{qG}, we recall some of the main results proved by T. Oikhberg, we establish a new result about $\n$-suppression-quasi-greedy bases with constant one, and we give a partial answer to an open question that connects $\n$-$t$-quasi-greedy and $\n$-$s$-quasi-greedy bases for $s,t\in (0,1]$. In Section \ref{adaptative}, we give a negative answer to the following question: given $x$ in a Banach space, is it possible to find a sequence $\n(x)=(n_1(x)<n_2(x)<\cdots)$ such that the TGA converges for this sequence?  In Section \ref{boundedgaps}, for Schauder bases we give the answer to another question posed in \cite{O2015}, related to sequences with bounded gaps. In Section \ref{section6}, we study subsequences of weakly null sequences, and in Section \ref{section7} we extend tha main result of \cite{O2015} to the context of quasi-Banach spaces. Finally, in Section \ref{section8}, we leave some open questions for future research in the area. 

\section{General setting}\label{intro}

Let $\X$ be a separable infinite dimensional Banach $\X$ space over the field  $\mathbb F=\mathbb R$ or $\mathbb C$, with a \textit{semi-normalized Markushevich basis}  $\mathcal B=(\be_i)_{i=1}^\infty$, that is, if $\X^*$ is the dual space of $\X$, $(\be_i)_{i=1}^\infty$ satisfies the following:

	\begin{enumerate}[\upshape (i)]
		\item $\X=\overline{[\be_i \colon i\in\N]}$,
		\item there is a (unique) sequence $(\be_i^*)_{i=1}^{\infty}\subset \X^*$, called  biorthogonal functionals, such that $\be_k^*(\be_i)=\delta_{k,i}$ for all $k,i\in\N$.
		\item  if $\be_i^*(x)=0$ for all $i$, then $x=0$.
		\item There is  a constant $c>0$ such that $\sup_i\lbrace\Vert \be_i\Vert, \Vert \be_i^*\Vert\rbrace\leq c$.
	\end{enumerate}
Additionally, if $\B$ verifies the above conditions and there is $C>0$ such that 
\begin{enumerate}[\upshape (i)]
	\item[(v)] $\Vert S_m(x)\Vert \leq C \Vert x\Vert$, for all $x\in\X$ and all $m\in\N$, where $S_m$ is the $m$th partial sum, we say that $\B$ is a \textit{Schauder basis}. Its \textit{basis constant} $\K$ is the minimum $C$ for which this inequality holds. 
\end{enumerate}
Hence, for a semi-normalized Markushevich basis $\mathcal B$ and $x\in\X$, we have the following formal decomposition $x=\sum_{i=1}^{\infty}\be_i^*(x)\be_i$, where $\lim_{i\rightarrow+\infty}\be_i^*(x)=0$ and the assignment of the coefficients is unique. 

Since 1999, one of the most important algorithms studied in the field of non-linear approximation is the Thresholding Greedy Algorithm (TGA) $(\G_m)_{m=1}^\infty$ introduced by S. V. Konyagin and V. N. Temlyakov in \cite{KT}.  Basically, for an element $x\in\X$, the algorithm selects the biggest coefficients of $x$ in modulus. In this occasion, we work with a relaxed version of the TGA introduced in \cite{Tem}. For a fix number $t\in (0,1]$, we say that a set $A=A(x)$ is a $t$-\textbf{greedy set} for $x\in\mathbb X$ if
$$\min_{i\in A}\vert\be_i^*(x)\vert\geq t\cdot \max_{i\not\in A}\vert\be_i^*(x)\vert.$$
Of course, if $\vert\supp(x)\vert=\infty$, the set $A$ is finite due to the fact that, for every $\varepsilon>0$, the set $\lbrace n\, :\, \vert\be_i^*(x)\vert>\varepsilon\rbrace$ is finite (see \cite{DKO}). A $t$-\textbf{greedy sum} is the projection
$$\mathbf{G}_m^t(x)=\sum_{i\in A}\be_i^*(x)\be_i,$$
where $A$ is a $t$-greedy set of cardinality $m$ for $x\in\X$. The collection $(\mathbf{G}^t_m)_{m=1}^\infty$ is called the \textbf{Weak Thresholding Greedy Algorithm} (WTGA) (see \cite{Tem,Tem2}), and we denote by $\mathcal G_m^t$ the collection of $t$-greedy sums $\G^t_m$ with $m\in\mathbb N$. If $t=1$, we  talk about greedy sets and greedy sums $\G_m$.

Different types of convergence of these algorithms have been studied in several papers, for instance \cite{DKK2003, DKKT, KT}. A central concept in this theory is the notion of quasi-greediness (\cite{KT}), where a basis is quasi-greedy if $\Vert\G_m(x)\Vert\lesssim \Vert x\Vert$, for all $x\in\X$ and for all $m\in\N$. The relation between quasi-greediness and the convergence of the algorithm was given by P. Wojtaszczyk in \cite{Wo}: a basis is quasi-greedy if and only 
\begin{eqnarray}\label{qg}
\lim_n \G_n(x)=x,\; \forall x\in\X.
\end{eqnarray}

Consider now a strictly increasing sequence of positive integers $\mathbf{n}=(n_1<n_2<\cdots)$. T. Oikhberg defined in \cite{O2015} the $\n$-$t$-quasi-greediness as follows: we say that $\mathcal B$ is $\n$-$t$-quasi-greedy if
\begin{eqnarray}\label{qO}
\lim_i \G_{n_i}^t(x)=x,
\end{eqnarray}
for any $x\in\X$ and any choice of $t$-greedy sums $\G_{n_i}^t(x)$.
Of course, in view of the inequalities \eqref{qg} and \eqref{qO}, if the basis is quasi-greedy then the basis is $\n$-quasi-greedy. One of the most important results in \cite{O2015} is the existence of $\n$-$t$-quasi-greedy bases that are not quasi-greedy. In fact, \cite[Proposition 3.1]{O2015} shows that there are such bases for any sequence $\n$ that has arbitrarily gaps, that is,
$$\limsup_i \frac{n_{i+1}}{n_i}=\infty, $$
whereas the question of whether that is possible for sequences that do not have such gaps is left open \cite[Question 1]{O2015}.

Throughout the paper, $\n=(n_i)_{i\in\N}$ denotes a strictly increasing sequence of natural numbers, and $\mathcal B$ denotes a semi-normalized Markushevich basis $(\be_i)_i$ in a Banach space $\X$ with biorthogonal functionals $(\be_i^*)_i$. We also define
$$
\alpha_1:=\sup_{i\in \N}\|\be_i\|\qquad \text{and}\qquad \alpha_2:=\sup_{i\in \N}\|\be_i^*\|.
$$
As usual, by $\supp{(x)}$ we denote the support of $x\in \X$, that is the set $\{i\in \mathbb{N}: \be_i^*(x)\not=0\}$, and $P_A$ with $A$ a finite set denotes the projection operator, that is,
$$P_A(x)=\sum_{i\in A}\be_i^*(x)\be_i, $$ with the convention that any sum over the empty set is zero (note tha if $P_A$ is uniformly bounded, $\B$ is an unconditional basis with constant $\K_s=\sup_{\vert A\vert<\infty}\Vert P_A\Vert$). \\
For $A$ and $B$ subsets of $\mathbb N$, we write $A<B$ to mean that  $\max A<\min B$. If $m\in \mathbb N$, we write $m < A$ and $A < m$ for $\{ m\} < A$ and  
$A <\{ m\}$ respectively (and we use the symbols ``$>$'', ``$\ge$'' and ``$\le$'' similarly). Also, $A\cupdot B$ means the union of $A$ and $B$ with $A\cap B=\emptyset$. Finally, we denote by $\N_{>1}=\N\setminus \lbrace 1\rbrace$. 

\section{$\n$-$t$-quasi-greedy bases}\label{qG}

As P. Wojtaszczyk did for quasi-greedy bases in \cite{Wo}, T. Oikhberg proved the following results connecting the convergence of the WTGA for general sequences $\n$ with the boundedness of the $t$-greedy sums.

\begin{theorem}{\cite[Theorem 2.1]{O2015}}\label{th}
	A basis $\mathcal B$ is $\mathbf{n}$-$t$-quasi-greedy if and only if 
	$$\Vert \mathbf{G}^t_n(x)\Vert\leq C\Vert x\Vert,\; \forall x\in\mathbb X, \forall \G_n^t\in\mathcal G_n^t, \forall n\in\mathbf{n}.$$
\end{theorem}

Although T. Oikhberg defined $\n$-$t$-quasi-greedy bases using the condition \eqref{qO}, using the above result, we use the following equivalent definition for our purposes. 
\begin{defi}
	We say that $\mathcal B$ is $\mathbf{n}$-$t$-\textbf{quasi-greedy} (for $t\in (0,1]$) if there exists a positive constant $C$ such that
	\begin{eqnarray}\label{q}
	\Vert \mathbf{G}^t_n(x)\Vert\leq C\Vert x\Vert,\; \forall x\in\mathbb X, \forall \G_n^t\in\mathcal G_n^t, \forall n\in\mathbf{n}.
	\end{eqnarray}
	Alternatively, we say that $\mathcal B$ is $\n$-$t$-\textbf{suppression quasi-greedy} if there exists a positive constant $C$ such that 
	\begin{eqnarray}\label{sq}
	\Vert x-\mathbf{G}^t_n(x)\Vert \leq C\Vert x\Vert,\; \forall x\in\mathbb{X}, \forall \G_n^t\in\mathcal G_n^t, \forall n\in\n.
	\end{eqnarray}
	We denote by $\C_{q,t}$ and $\C_{sq,t}$ the smallest constants verifying \eqref{q} and \eqref{sq}, respectively, and we say that $\mathcal B$ is $\C_{q,t}$-$\n$-$t$-quasi-greedy and $\C_{sq,t}$-$\n$-$t$-suppression-quasi-greedy.
\end{defi}

\begin{remark}
	For $\n=\mathbb N$ and $t=1$, we recover the classical definitions of quasi-greediness (see \cite{KT}).\end{remark}

We will use the following notation for $\n$-$t$-quasi-greedy bases (resp. $\n$-$t$-suppression quasi-greedy bases):
\begin{itemize}
	\item If $\n=\mathbb N$, we say that $\mathcal B$ is $\C_{q,t}$-$t$-quasi-greedy.
	\item If $t=1$, we say that $\mathcal B$ is $\C_{q}$-$\n$-quasi-greedy.
	\item If $t=1$ and $\n=\mathbb N$, we say that $\mathcal B$ is $\C_{q}$-quasi-greedy.
\end{itemize}

In the literature, several papers study greedy-type bases with constant 1 such as \cite{AW, AA2017, DKOSZ}. For instance, in \cite{AA2017}, the authors proved the following result.
\begin{theorem}[\cite{AA2017}]
	Let $\mathbb X$ be a Banach space and $\mathcal B$ a basis in $\mathbb X$. $\mathcal B$ is $1$-unconditional if and only if $\mathcal B$ is $1$-quasi-greedy.
\end{theorem}

An equivalent result for general sequences $\n$ and $t=1$ is given in \cite{O2015}. 

\begin{proposition}{\cite[Proposition 2.7]{O2015}}
	A basis $\mathcal B$ is $1$-$\n$-quasi-greedy basis if and only if $\mathcal B$ is $1$-suppression-unconditional
\end{proposition}
A consequence of this result is that every $1$-$\n$-$t$-quasi-greedy basis is quasi-greedy. Here, we prove the same consequence for $1$-$\n$-$t$-suppression-quasi-greedy bases.
\begin{theorem}
	Let $\mathcal B$ be a $1$-$\n$-$t$-suppression-quasi-greedy basis. Then, $\mathcal B$ is quasi-greedy. 
\end{theorem}
\begin{proof}
	Since a $1$-$\n$-$t$-suppression-quasi-greedy basis is $1$-$\n$-suppression-quasi-greedy, we assume $t=1$. Let 
	$$M:=n_1 \alpha_1\alpha_2.$$ 
	It is immediate that for all $A\subset \N$ with $|A|\le n_1$, 
	\begin{equation}
	\|P_A\|\le M,\qquad \text{ and } \qquad \|I-P_A\|\le M+1. \label{lessthann1}
	\end{equation}
	We claim that $\B$ is $M+1$-suppression quasi-greedy. Otherwise, there is a minimum $m\in \N$ for which there is $x\in \X$ with $\|x\|=1$ and a greedy sum for $x$ of order $m\in\mathbb N$ such that 
	$$
	\|x-\G_m(x)\|>M+1. 
	$$
	It follows by \eqref{lessthann1} that $m>n_1$. Then 
	$$
	M+1<\|x-\G_m(x)\|=\|x-\G_{m-n_1}(x)-\G_{n_1}(x-\G_{m-n_1}(x))\|\le \|x-\G_{m-n_1}(x)\|,
	$$
	contradicting the minimality of $m$. 
\end{proof}

To close this section, we will discuss the open question \cite[Question 6.4]{O2015}. It follows at once from the definitions that if a basis in $\n$-$t$-quasi-greedy for some $0<t\leq1$, then it is also $\n$-$s$-quasi-greedy for all $t<s\leq 1$. The author asks under which conditions the converse is true. While we do not have a full answer, the next lemma gives sufficient conditions.

\begin{lemma}
	Let $0<s\leq 1$ and let $\n=(n_i)_{i\in\N}$ be a strictly increasing sequence of natural numbers. If $\mathcal B$ is a $\C_{q,s}$-$\n$-$s$-quasi-greedy basis, then, for each $t$ with
	\begin{eqnarray}\label{bound1b}
	s\left(1-\frac{1}{\C_{q,s}}\right)<t<s,
	\end{eqnarray}
	$\mathcal B$ is $\n$-$t$-quasi-greedy with constant
	$$\C_{q,t}\leq \frac{\C_{q,s}t}{s-\C_{q,s}(s-t)}.$$
\end{lemma}
\begin{proof}
	
	Fix $x\in \mathbb{X}$, $n\in \mathbf{n}$ and $A$ a $t$-greedy set for $x$ of cardinality $n$. Let 
	$$
	y:=P_{A}(x)+ \frac{t}{s}(x-P_{A}(x)).
	$$
	For each $i\in A$ and each $j\not \in A$, we have
	$$
	|\be_i^*(y)|=|\be_i^{*}(x)|\ge t |\be_j^{*}(x)|= s|\be_j^{*}(y)|.
	$$
	Thus, $A$ is an $s$-greedy set for $y$ of cardinality $n$. It follows that 
	$$
	\|P_{A}(x)\|=\|P_{A}(y)\|\le \C_{q,s}\|y\|\le  \C_{q,s} \left(\frac{t}{s}\|x\|+ (1-\frac{t}{s})\|P_{A}(x)\|\right).
	$$
	By \eqref{bound1b},
	$$
	1- \C_{q,s} \left(1-\frac{t}{s}\right)>0, 
	$$
	so this gives
	$$
	\|P_{A}(x)\|\le \frac{ \C_{q,s} t}{s- \C_{q,s}(s-t)}\|x\|.
	$$
\end{proof}

\section{Adaptive gaps for Markushevich bases.}\label{adaptative}
In this section, we address \cite[Question 6.5]{O2015} regarding adaptive gaps. Concretely, we give a negative answer to this question by means of a new example. The question is as follows:

\begin{center}
	For a Markushevich basis $\B$ in a Banach space $\X$ and $x\in\mathbb X$, can we find a sequence $\n(x)=(n_1(x)<n_2(x)<...)\subset\mathbb N$ so that the sequence $\lim_k \G_{n_k(x)}(x)=x$?
\end{center}

The following example shows that it is not possible, in general, to find this sequence since we find an element in a Banach space such that \eqref{qO} does not hold for any sequence $\n$.
\begin{example}\label{ex:question 5}
	Let $\mathbb{X}$ be the completion of $\mathtt{c}_{00}$ under the norm
	$$
	\Vert \mathbf (a_i)_i\Vert:= \sup_{n\geq 1}\Big\vert \sum_{j=1}^n a_j\Big\vert,
	$$
and let $\mathcal B=(\be_i)_{i\in\N}$ the canonical basis. There is $y\in \mathbb{X}$ with the property that for every $0<t\le 1$,
	\begin{equation}
	\inf_{\G_m^t(y)\in \mathcal G_m^t(y)}\|\G_m^t(y)\|\xrightarrow[m\to +\infty]{}  +\infty. \label{toinfinity}
	\end{equation}
	In particular, taking $t=1$ it follows that for any strictly increasing sequence of integers $\{m_j\}_j$ and every sequence of greedy sums $\{\G_{m_j}(y)\}_j$, we have
	$$
	\|\G_{m_j}(y)-y\|\xrightarrow[j\to +\infty]{}  +\infty. 
	$$
\end{example}
\begin{proof}
	First, note that the canonical basis $(\be_i)_i$ is a monotone normalized Schauder basis for $\X$, and for every scalar sequence $(a_i)_i$, there is $x\in \X$ such that 
	$$
	x=\sum\limits_{i=1}^{\infty}a_i\be_i
	$$
	if and only if the scalar series $\sum\limits_{i=1}^{\infty}a_i$ converges.   \\
	Next, define a sequence of positive integers $(n_k)_k$ and a sequence of finite sets of positive integers $(A_k)_k$ as follows: $n_1:=1$, and for all $k\in \N$, 
	$$
	n_{k+1}:=n_k+10^{k}+1, \qquad A_k:=\{n_{k}+1,\dots, n_k+10^{k}\}.
	$$
	It is clear that $n_k<A_k<n_{k+1}$, $|A_k|=10^{k}$ for all $k$, and
	$$
	\N=\bigcup_{k}\{n_k\}\bigcup_{k}A_k
	$$
	Now define a scalar sequence $(b_n)_n$ as follows: 
	\begin{numcases}{b_n:=}
	\frac{1}{\sqrt{k}} & if $n=n_k$,\nonumber\\
	-\frac{1}{10^{k}\sqrt{k}} & if $n\in A_k$. \nonumber
	\end{numcases}
	With these definitions (as usual, the empty sum is zero), it follows that for all $k\in \N$,
	\begin{eqnarray}
	\sum\limits_{1\le n<n_k}b_n&=&\sum\limits_{1\le j<k}b_{n_j}+\sum\limits_{1\le j<k}\sum\limits_{i\in A_j}b_i=	\sum_{1\le j<k}\frac{1}{\sqrt{j}}+\sum_{1\le j<k}\sum_{i\in A_j}\frac{(-1)}{10^j\sqrt{j}}\nonumber\\
	&=&\sum_{1\le j<k}\frac{1}{\sqrt{j}}-\sum_{1\le j<k}\frac{1}{\sqrt{j}}=0.
	\end{eqnarray}
	Thus, 
	\begin{eqnarray}
	\sum\limits_{n=1}^{n_k}b_n=b_{n_k}=\frac{1}{\sqrt{k}},\label{oneofthen_k}
	\end{eqnarray}
	and for $l\in A_k$,
	\begin{eqnarray}
	\left|\sum\limits_{n=1}^{l}b_n\right|=\left|\sum\limits_{n=1}^{n_k}b_n+\sum\limits_{\substack{n\in A_k\\ n\le l}}b_n\right|=\left|\frac{1}{\sqrt{k}}-\frac{1}{10^{k}\sqrt{k}}|\{n\in A_k: n\le l\}|\right|< \frac{1}{\sqrt{k}}.\label{inAk}
	\end{eqnarray}
It follows from this that $\sum\limits_{i=1}^{\infty}b_i=0$, so there is $y\in \X$ given by 
	$$y:=\sum\limits_{i=1}^{\infty}b_i\be_i,$$
	and \eqref{oneofthen_k} and \eqref{inAk} entail that $\|y\|=1$. 
	We claim that \eqref{toinfinity} holds. To prove our claim, suppose it is false. Then, for some $0<t\le 1$, there is $M>0$ and a strictly increasing sequence of integers $\{m_j\}_j$ such that 
	\begin{eqnarray*}
	\inf_{\G_{m_j}^t(y)\in \mathcal{G}_{m_j}^t(y)}\|\G_{m_j}^t(y)\|<M. 	
	\end{eqnarray*}
	Choose a sequence of $t$-greedy sums $\{\G_{m_j}^t(y)\}_j$ so that 
	\begin{equation}
	\|\G_{m_j}^t(y)\|<M \label{bounded}
	\end{equation}
	for all $j$. For each $j\in \N$, the support of $\G_{m_j}^t(y)$ is a $t$-greedy set, say $\Lambda_{m_j}^t(y)$. Define a function $\phi:\N\rightarrow \N$ by 
	$$
	\phi(j):=\min\{k\in \N: n_k\not \in \Lambda_{m_j}^{t}(y)\}. 
	$$
	By definition, for every $j\in \N$ and every $k\le \phi(j)-1$, $n_k\in \Lambda_{m_j}^t(y)$. Hence, 
	\begin{eqnarray}
	\|\G_{m_j}^t(y)\|&=&\sup_{i\ge 1}\left|\sum\limits_{l=1}^{i}\be_l^*(\G_{m_j}^t(y))\right|\ge\left|\sum\limits_{i=1}^{\max\{\Lambda_{m_j}^t(y)\}}\be_i^*(\G_{m_j}^t(y))\right|=\left|\sum\limits_{i\in \Lambda_{m_j}^t(y)}\be_i^{*}(y)\right|\nonumber\\
	&\ge& \sum\limits_{i\in \Lambda_{m_j}^t(y)}\be_i^{*}(y)=\sum\limits_{\substack{k\in \N\\n_k\in  \Lambda_{m_j}^t(y)}}\frac{1}{\sqrt{k}} +\sum_{\substack{k\in\N\\A_k\cap \Lambda_{m_j}^t(y)\not=\emptyset}}\sum\limits_{i\in A_k\cap \Lambda_{m_j}^t(y)} \be_i^*(y)\nonumber\\
	&\ge& \sum\limits_{k=1}^{\phi(j)-1}\frac{1}{\sqrt{k}}+\sum_{\substack{k\in\N\\A_k\cap \Lambda_{m_j}^t(y)\not=\emptyset}}\sum\limits_{i\in A_k\cap \Lambda_{m_j}^t(y)} \be_i^*(y)\nonumber\\
	&=&\sum\limits_{k=1}^{\phi(j)-1}\frac{1}{\sqrt{k}}-\sum_{\substack{k\in\N\\A_k\cap \Lambda_{m_j}^t(y)\not=\emptyset}}\frac{|A_k\cap \Lambda_{m_j}^t(y)|}{10^k\sqrt{k}}\nonumber\\
	&\ge& \sum\limits_{k=1}^{\phi(j)-1}\frac{1}{\sqrt{k}}-\sum_{\substack{k\in\N\\A_k\cap \Lambda_{m_j}^t(y)\not=\emptyset}}\frac{1}{\sqrt{k}}.\label{intermediate}
	\end{eqnarray}
	For each $j\in \N$, $n_{\phi(j)}\not \in \Lambda_{m_j}^{t}(y)$. Thus, the $t$-greedy condition implies that for every $n\in \Lambda_{m_j}^{t}(y)$,
	\begin{equation}
	|\be_i^*(y)|\ge t |\be_{n_{\phi(j)}}^*(y)|=\frac{t}{\sqrt{\phi(j)}}.\label{boundforpsi}
	\end{equation}
	This implies that for all $k,j\in \N$, if $\Lambda_{m_j}^t(y)\cap A_k\not=\emptyset$, then for each $n\in \Lambda_{m_j}^t(y)\cap A_k$, 
	\begin{equation}
	\frac{1}{10^k}\ge \frac{1}{10^k\sqrt{k}}=|\be_n^*(y)|\ge \frac{t}{\sqrt{\phi(j)}}, \nonumber
	\end{equation}
	so
	\begin{equation*}
	10^k\le \frac{\sqrt{\phi(j)}}{t}.
	\end{equation*}
	Hence, 
	$$
	k\le \log_{10} \frac{\sqrt{\phi(j)}}{t},
	$$
	or (equivalently since $k\in \N$),
	\begin{equation}
	k\le  \floor*{ \log_{10}\frac{\sqrt{\phi(j)}}{t}}.\label{boundfortheAksinLambdam}
	\end{equation}
	Thus, for all $j\in \N$,
	$$
	\sum_{\substack{k\in\N\\A_k\cap \Lambda_{m_j}^t(y)\not=\emptyset}}\frac{1}{\sqrt{k}}\le \sum\limits_{k=1}^{\floor*{\log_{10} \frac{\sqrt{\phi(j)}}{t}}}\frac{1}{\sqrt{k}}.
	$$
	From this and \eqref{intermediate} it follows that 
	\begin{equation}
	\|\G_{m_j}^t(y)\|\ge \sum\limits_{k=1}^{\phi(j)-1}\frac{1}{\sqrt{k}}-\sum\limits_{k=1}^{\floor*{\log_{10} \frac{\sqrt{\phi(j)}}{t}}}\frac{1}{\sqrt{k}}\label{intermediate2}
	\end{equation}
	for all $j\in\N$. \\
	Now since $(\be_i^*)_i$ is $w^*$-null, $t$ is fixed, and for every $j\in \N$ there is $n\in \Lambda_{m_j}^t(y)\cap [m_j,+\infty)$, \eqref{boundforpsi} implies that
	\begin{equation}
	\phi(j)\xrightarrow [j\to +\infty]{}+\infty. \nonumber%\label{infinity}
	\end{equation}
	Therefore, there is $j_0$ such that for all $j\ge j_0$,
	\begin{equation}
	\sqrt{\phi(j)-1}-\sqrt{\floor*{\log_{10} \frac{\sqrt{\phi(j)}}{t}}+1}>\frac{M}{2}.  \label{sufficientlylarge}
	\end{equation}
	It follows from \eqref{intermediate2} and \eqref{sufficientlylarge} that for all $j\ge j_0$,
	\begin{eqnarray*}
		\|\G_{m_j}^t(y)\|&\ge& \sum\limits_{\floor*{\log_{10} \frac{\sqrt{\phi(j)}}{t}}+1}^{\phi(j)-1}\frac{1}{\sqrt{k}}\ge \int_{\floor*{\log_{10} \frac{\sqrt{\phi(j)}}{t}}+1}^{\phi(j)-1}\frac{1}{\sqrt{t}}dt\\
		&=&2\left(\sqrt{\phi(j)-1}-\sqrt{\floor*{\log_{10} \frac{\sqrt{\phi(j)}}{t}}+1}\right)>M. 
	\end{eqnarray*}
	As this contradicts \eqref{bounded}, the proof is complete. 
\end{proof}

\section{$\n$-$t$-quasi-greedy bases for sequences with bounded gaps}\label{boundedgaps}
In \cite[Proposition 3.1]{O2015}, it is proven that for every sequence $\mathbf{n}=(n_i)_{i\in\N}$ with arbitrarily large gaps, there is a space $\mathbb{X}$ with a (Schauder) basis $\mathcal B$ such that $\B$ is $\mathbf{n}$-$t$-quasi-greedy for all $0<t\le 1$, but not quasi-greedy. Recall that $\n$ has arbitrarily large gaps if
$$\limsup_i \frac{n_{i+1}}{n_i}=\infty.$$
An open question posed in the same paper, concretely \cite[Question 6.1]{O2015}, says the following:
\begin{center}
	Are there $\mathbf{n}$-quasi-greedy but not quasi-greedy bases such that $\n$ does not have arbitrarily large gaps?
\end{center}
To show that the answer of this question is negative for Schauder bases, we use the following definition.
\begin{defi}
	Let $\n=(n_k)_{k\in \N}$ be a strictly increasing sequence of natural numbers, and let $l\in\mathbb N_{>1}$. We say that $\n$ has $l$-bounded gaps if
	$$\frac{n_{i+1}}{n_i}\leq l,$$
	for all $i\in\mathbb N$, and we say that it has bounded gaps if it has $l$-bounded gaps for some $l\in\mathbb N_{>1}$.
\end{defi}

\begin{theorem}\label{theoremboundedgapsv2} Let $\mathbf{n}=(n_k)_k$ be a strictly increasing sequence of natural numbers and $0<t\le 1$. Suppose that $\mathcal B$ is a $\C_{q,t}$-$t$-$\mathbf{n}$-quasi-greedy Schauder basis with basis constant $\K$. For every $k\in \N$, $l\in \N_{>1}$, $x\in \X$ and $A$ any $t$-greedy set for $x$ with $n_k\le |A|< l\cdot n_k$, we have 
	\begin{equation}
	\|P_{A}(x)\|\le 2\C_{q,t}\K(l-1+\K)\|x\|.\label{largercardinals2}
	\end{equation}
	In particular, if $\mathbf{n}$ has bounded gaps, the basis is quasi-greedy. 
\end{theorem}
\begin{proof}
	To prove \eqref{largercardinals2}, we assume without loss of generality that $n_k<|A|<ln_k$. Thus, there is $2\le j\le l$ and a partition of $A$ into $j$ disjoint nonempty sets $(A_i)_{1\le i\le j}$ with the following properties:  
	\begin{equation}
	|A_1|\le n_{k}, \qquad |A_i|=n_{k}\;\forall 2\le i\le j,\qquad A_i<A_{i+1}\;\forall 1\le i\le j-1. \label{firstsets2}
	\end{equation}
	For each $2\le i\le j$, define 
	$$
	B_i:=\{\min{A_i},\dots, \max{A_i} \},
	$$
	Observe that for each $2\le i\le j$, $A_i$ is a $t$-greedy set for $P_{B_i}(x)$ of cardinality $n_k$. Hence, 
	$$
	\|P_{A_i}(x)\|=\|P_{A_i}(P_{B_i}(x))\|\le \C_{q,t}\|P_{B_i}(x)\|. 
	$$
	Now for each $2\le i\le j$, $B_i$ is an interval of positive integers.  Thus, the Schauder condition implies that 
	$$
	\|P_{B_i}(x)\|\le 2\K\|x\|. 
	$$
	Consequently, we have
	\begin{equation}
	\|P_{A_i}(x)\|\le 2\C_{q,t}\K\|x\|.\label{fori>12}
	\end{equation}
	for all $2\le i\le j$. If $|A_1|=n_{k}$, then the same argument gives that \eqref{fori>12} holds also for $i=1$. If $0<|A_1|<n_{k}$, let $\widetilde{A}_1$ be the set consisting of the first $n_{k}-|A_1|$ elements of $A\setminus A_1$. As $A_1<\widetilde{A}_1$, the Schauder condition entails that
	\begin{equation}
	\|P_{A_1}(x)\|\le \K\|P_{A_1 \cup \widetilde{A}_1}(x)\|.\label{A12}
	\end{equation}
	Define
	$$
	B_1:=\{\min{A_1},\dots, \max{\widetilde{A}_1} \}.
	$$
	It is clear that $B_1=\{\min{A_1\cup \widetilde{A}_1},\dots, \max{A_1 \cup \widetilde{A}_1} \}$. Since $A_1 \cup \widetilde{A}_1$ is a $t$-greedy set for $P_{B_1}(x)$ of cardinality $n_k$, we have
	\begin{equation}
	\|P_{A_1 \cup \widetilde{A}_1}(x)\|\le \C_{q,t}\|P_{B_1}(x)\|. \label{B12}
	\end{equation}
	Finally, the fact that $B_1$ is an interval of positive integers entails that 
	$$
	\|P_{B_1}(x)\|\le 2\K\|x\|.
	$$
	This bound, when combined with \eqref{A12} and \eqref{B12}, gives that 
	$$
	\|P_{A_1}(x)\|\le 2\C_{q,t}\K^2\|x\|. 
	$$
	From this result and \eqref{fori>12}, it follows that 
	\begin{equation}
	\|P_A(x)\|\le \sum\limits_{i=1}^{j}\|P_{A_i}(x)\|\le ((j-1)2\C_{q,t}\K+2C_{q,t}\K^2)\|x\|\le 2\C_{q,t}\K(l-1+\K)\|x\|, \nonumber
	\end{equation}
	so the proof of \eqref{largercardinals2} is complete. \\
	Suppose now that $\mathbf{n}$ has $l$-bounded gaps, and fix $x\in \X$, and $A$ a greedy set for $x$. We may assume without loss of generality that $|A|\not \in \mathbf{n}$ (else, there is nothing to prove). If $|A|<n_1$, then
	\begin{equation}
	\|P_A(x)\|\le \sum\limits_{i\in A}|\be_i^*(x)|\|\be_i\|\le n_1\alpha_1\alpha_2 \|x\|. \label{smallcardinal0}
	\end{equation}
	Otherwise, let 
	$$
	k_0:=\max_{k\in \N}\{n_k<|A|\}. 
	$$
	As $\mathbf{n}=(n_k)_k$ has $l$-bounded gaps and $|A|\not \in \mathbf{n}$, we have $n_{k_0}<|A|<l n_{k_0}$. Hence, by \eqref{largercardinals2}, 
	$$
	\|P_A(x)\|\le 2\C_{q,t} \K(l-1+\K)\|x\|.
	$$
	From this and \eqref{smallcardinal0} it follows that $\B$ is $\mathbf C$-quasi-greedy with 
	$$
	\mathbf C\le \max\{n_1\alpha_1\alpha_2, 2\C_{q,t} \K(l-1+\K)\}.
	$$
\end{proof}

\section{Subsequences of weakly null sequences.}\label{section6}

In \cite[Theorem 5.4]{DKK2003}, it was proven that every weakly null semi-normalized sequence with a spreading model not equivalent to the unit vector basis of $\mathtt{c_0}$, has a quasi-greedy (in fact, almost greedy) basic subsequence. The question of whether every weakly null semi-normalized sequence has a quasi-greedy basic subsequence remains open, but it was proven in \cite[Corollary 4.6]{DKSWo2012} that if the sequence is a branch quasi-greedy basis, it has a quasi-greedy subsequence. Using  \cite[Theorem 5.4]{DKK2003}, we can prove a similar result for $\mathbf{n}$-quasi-greedy bases. It is known that every weakly null semi-normalized sequence has a subsequence with a $1$-suppression unconditional spreading model (see for instance \cite{AK2006}). Thus, by \cite[Theorem 5.4]{DKK2003} we only need to consider cases in which the spreading model is equivalent to the unit vector basis of $\mathtt{c}_0$. It turns out that in such cases, an $\mathbf{n}$-quasi-greedy basis is already quasi-greedy.

\begin{proposition}\label{propositionlikec0}Assume that $\B$  is a $\C_q$-$\mathbf{n}$-quasi-greedy basis, and $(\be_{i_k})_k$ is a subsequence with a spreading model $(y_i)_i$ that is $M$-equivalent to the unit vector basis of $\mathtt{c}_0$. Then $\B$ is $K_q$-quasi-greedy with 
	\begin{equation*}
	K_q\le \C_q+(\C_q+1)\alpha_2M.
	\end{equation*}
\end{proposition}
\begin{proof}
	Fix $x\in \X$ with finite support, $m\in \N$, and $A$ a greedy set for $x$ of cardinality $m$. Choose $\epsilon>0$, $n\in \mathbf{n}\cap [m+1,+\infty)$, and let 
	$$
	\Y:=\overline{[y_i:i\in \N]}.
	$$
	By the spreading property and the hypothesis on $(y_i)_i$, there is $k_0\in \N$ such that for each $B\subset\{i_k\}_k$ with $B>i_{k_0}$ and $|B|=n-m$, 
	\begin{equation}
	\|\bff_B\|\le (1+\epsilon)\|\sum\limits_{i=1}^{n-m}y_i\|_{\Y}\le (1+\epsilon)M. \label{farenough}
	\end{equation}
	Choose $B$ as above, with the additional condition that $B>\supp{(x)}$, and let 
	$$
	y:=x+\min_{i\in A}|\be_i^*(x)|\bff_B.
	$$
	As $A\cup B$ is a greedy set for $y$ with cardinality $n\in\mathbf{n}$, by the $\C_q$-$\mathbf{n}$-quasi-greedy property, \eqref{farenough} and the triangle inequality it follows that 
	\begin{eqnarray}
	\|P_A(x)\|&=&\|P_{A\cup B}(y)-\min_{i\in A}|\be_i^*(x)|\bff_B\|\le \|P_{A\cup B}(y)\| +\min_{i\in A}|\be_i^*(x)|\|\bff_B\|\nonumber\\
	&\le& \C_q\|y\|+\min_{i\in A}|\be_i^*(x)|\|\bff_B\|\nonumber\\
	&\le& \C_q\|x\|+(\C_q+1)\min_{i\in A}|\be_i^*(x)|\|\bff_B\|\nonumber\\
	&\le& \C_q\|x\|+(\C_q+1)\alpha_2(1+\epsilon)M\|x\|.
	\end{eqnarray}
	By letting $\epsilon\rightarrow 0$, the proof for $x$ with finite support is complete. The general case follows by  \cite[Corollary 2.3]{O2015}.
\end{proof}

\begin{corollary}\label{corollaryweaklynull}Assume that $\B$ is $\mathbf{n}$-quasi-greedy. If $\B$ is weakly null, it has a quasi-greedy subsequence. 
\end{corollary}
\begin{proof}
	By passing to a subsequence if necessary, we may assume that $\B$ is a Schauder basis with an unconditional spreading model $(y_i)_i$. If $(y_i)_i$ is equivalent to the unit vector basis of $\mathtt{c}_0$, then $\B$ is quasi-greedy by Proposition~\ref{propositionlikec0}. Else, $\B$ has a quasi-greedy subsequence by \cite[Theorem 5.4]{DKK2003}. 
\end{proof}
\section{The case of Quasi-Banach spaces}\label{section7}

Recently, in \cite{AABW, B}, the authors have extended the classical results of greedy-type bases to the context of quasi-Banach spaces studying how the lack of convexity affects the main results involving characterizations of quasi-greedy, almost greedy and greedy bases. As the topic of greedy-type bases in quasi-Banach spaces is being developed and there are several researchers focusing their attention on the context of quasi-Banach spaces, in this section we extend the main result of T. Oikhberg (that is, Theorem \ref{th}).

We use $\alpha$ to denote a constant for which $\|x+y\|\le \alpha\|x\|+\alpha\|y\|$ for all $x,y$ in a quasi-Banach space $\X$.  Also, for a semi-normalized basis $\B=(\be_i)_i$ with coordinate functionals $(\be_i^*)_i$, we define 
$$
c:=\sup\limits_{i\in \N}(1+\|\be_i\|)(1+\|\be_i^*\|). 
$$
\begin{remark}\label{remarkeasypart}Note that $c> 2$ and that for each finite set $A\subset \N$ and every $x\in \X$, 
	$$
	\|P_{A}(x)\|=\|\sum\limits_{i\in A}\be_i^*(x)\be_i\|\le \alpha^{|A|-1}\sum\limits_{i\in A}\|\be_i^*(x)\be_i\|\le \alpha^{|A|-1}|A|c\|x\|.
	$$
\end{remark}

First, we extend \cite[Lemma 2.2]{O2015} since is unknown in quasi-Banach spaces and the proof is a little bit different.

\begin{lemma}\label{lemma2.1extensiontrivial} Let $\B$ be a basis for a quasi-Banach space $\X$. If $A$ is a $t$-greedy set for $x$, for every $\epsilon>0$ there is $y$ with finite support such that $\|x-y\|\le \epsilon$ and $A$ is a $t$-greedy set for $x$. 
\end{lemma}
\begin{proof}
	The proof is just the proof of Lemma~\cite[Lemma 2.1]{O2015}, with minor modifications. We may assume that $A\not=\emptyset$ and $x$ does not have finite support. Let 
	$$
	\delta:=\frac{\epsilon}{4c^2\alpha^{|A|}|A|},
	$$
	and choose $z$ with finite support so that 
	\begin{equation}
	\|x-z\|\le \delta. \nonumber%\label{ledelta}
	\end{equation}
	Let
	\begin{equation}
	y:=z+ 2 c \delta \sum\limits_{i\in A}\sgn{\be_i^*(x)}\be_i. \nonumber%\label{closeenough5}
	\end{equation}
	We have
	\begin{eqnarray}
	\|x-y\|&\le& \alpha\|x-z\|+\alpha\|z-y\|\le \frac{\epsilon}{4}+ 2c \delta\alpha\alpha^{|A|-1}\sum\limits_{i\in A}\|\sgn{\be_i^*(x)}\be_i\|\nonumber\\
	&\le&\frac{\epsilon}{4}+2c^2\delta\alpha^{|A|}|A|= \frac{\epsilon}{4}+\frac{\epsilon}{2}<\epsilon.\nonumber%\label{x-y}
	\end{eqnarray}
	Note that 
	\begin{numcases}{\be_i^*(y)=}
	\be_i^*(x)+\be_i^*(z-x) & if $i\not \in A$,\label{notinA}\\
	\left(|\be_i^*(x)|+2c\delta \right)\sgn{\be_i^*(x)}+\be_i^*(z-x) & if $i\in A$.\label{inA}
	\end{numcases}
	Let 
	$$
	\beta:=\min_{i\in A}|\be_i^*(x)|. 
	$$
	For $i\not \in A$, by \eqref{notinA} we have
	$$
	|\be_i^*(y)|\le |\be_i^*(x)|+|\be_i^*(z-x)|\le \frac{\beta}{t}+c\|z-x\|\le\frac{\beta+ct\delta}{t}\le \frac{\beta+c\delta}{t}.
	$$
	On the other hand, \eqref{inA} entails that for $i\in A$, 
	$$
	|\be_i^*(y)|\ge |\be_i^*(x)|+2c\delta -|\be_i^*(z-x)|\ge \beta+2c\delta-c\|z-x\|\ge \beta+c\delta.
	$$
	Thus, $A$ is a $t$-greedy set for $y$. 
\end{proof}

Thanks to this result, we have the following result that many times, in Banach spaces, is used in the literature as a ``small perturbations argument".
\begin{corollary}\label{corollaryfinite} Let $\B$ a basis for a quasi-Banach space $\X$. Suppose there is $0<t\le 1$ and $\C>0$ such that for all $x\in \X$ with finite support, if  $A$ is a $t$-greedy set for $x$ with $|A|\in \mathbf{n}$, then
	\begin{equation}
	\|P_A(x)\|\le \C\|x\|.\label{hypothesisfinite}
	\end{equation}
	Then, for all $x\in \X$, if  $A$ is a $t$-greedy set for $x$ with $|A|\in \mathbf{n}$, 
	\begin{equation}
	\|P_A(x)\|\le \alpha^2\C\|x\|.\label{hypothesisforall}
	\end{equation}
\end{corollary}

\begin{proof}
	Suppose the implication is false, so \eqref{hypothesisfinite} holds for all $x$ with finite support, but there is $x\in \X$ with infinite support and $A$ a $t$-greedy set for $x$ with $|A|\in\mathbf{n}$ such that
	\begin{equation*}
	\|P_A(x)\|>\alpha^2\C\|x\|. 
	\end{equation*}
	Let 
	\begin{equation}
	\delta=\|P_A(x)\|-\alpha^2\C\|x\|,\qquad \epsilon:=\frac{\delta}{4\alpha^{|A|+2}c|A|\C }.\label{forcontradiction}
	\end{equation}
	As $\epsilon>0$, by Lemma~\ref{lemma2.1extensiontrivial} there is $y\in \X$ with finite support such that $\|x-y\|\le \epsilon$ and $A$ is $t$-greedy set for $y$. By \eqref{hypothesisfinite}, \eqref{forcontradiction} and Remark~\ref{remarkeasypart} we have
	\begin{eqnarray}
	\|P_A(x)\|&\le& \alpha\|P_A(y)\|+\alpha\|P_A(x-y)\|\le \alpha \C\|y\|+\alpha\alpha^{|A|-1}|A|c\|x-y\|\nonumber\\
	&\le &\alpha \C(\alpha\|x\|+\alpha\|x-y\|)+\alpha^{|A|}|A|c\|x-y\|\nonumber\\
	&\le& \alpha^2\C\|x\|+\alpha^2\C\epsilon+\alpha^{|A|}|A|c\epsilon\nonumber\\
	&=&\alpha^2\C\|x\|+\frac{\alpha^2\C\delta}{4\alpha^{|A|+2}|A|Kc}+\frac{\alpha^{|A|}|A|c\delta}{4\alpha^{|A|+2}c|A|\C}\le \alpha^2\C\|x\|+\frac{\delta}{4}+\frac{\delta}{4}\nonumber\\
	&=&\alpha^2\C\|x\|+\frac{\delta}{2}=\|P_A(x)\|-\delta+\frac{\delta}{2}\nonumber\\
	&<&\|P_A(x)\|. \nonumber
	\end{eqnarray}
	As this is absurd, the proof is complete. 
\end{proof}

Now we prove a variant of \cite[Theorem 2.1]{O2015}. 

\begin{theorem}\label{theoremqbextension}Let $0<t\le 1$, $\mathbf{n}=(n_k)_{k\in\N}$ a strictly increasing sequence of positive integers, and $\B$ a basis for a quasi-Banach space $\X$. The following are equivalent: 
	\begin{enumerate}[\upshape (i)]
		\item \label{bound} There is $\C_1>0$ such that for all $x\in \X$, if  $A$ is a $t$-greedy set for $x$ with $|A|\in \mathbf{n}$, then
		\begin{equation}
		\|P_A(x)\|\le \C_1\|x\|.\nonumber
		\end{equation}
		\item \label{boundfinitesupport} There is $\C_2>0$ such that for all $x\in \X$ with finite support, if  $A$ is a $t$-greedy set for $x$ with $|A|\in \mathbf{n}$, then
		\begin{equation}
		\|P_A(x)\|\le \C_2\|x\|.\nonumber
		\end{equation}
		\item \label{boundfinite>m} There is $m\in \N$ and $\C_3>0$ such that for all $x\in\X$ with finite support, if $m<\supp{(x)}$ and $A$ is a $t$-greedy set for $x$ with $|A|\in \mathbf{n}$, then 
		\begin{equation}
		\|P_A(x)\|\le \C_3\|x\|. \nonumber
		\end{equation}
		\item \label{pointwisebounded} For each $x\in \X$ there is $\C_4(x)>0$ such that if $A$ is a $t$-greedy set for $x$ and $|A|\in \mathbf{n}$, then 
		\begin{equation}
		\|P_A(x)\|\le \C_4(x). \nonumber
		\end{equation}
		\item \label{ntqg} For any $x\in\X$ and any choice of $t$-greedy sums $\G_{n_i}^t(x)$,
		$$\lim_{i}\G_{n_i}^t(x)=x.$$
	\end{enumerate}
\end{theorem}
\begin{proof}
	\eqref{bound}$\Longrightarrow$ \eqref{boundfinitesupport} $\Longrightarrow$ \eqref{boundfinite>m},  \eqref{bound}$\Longrightarrow$ \eqref{pointwisebounded} and  \eqref{ntqg}$\Longrightarrow$ \eqref{pointwisebounded} are immediate.\\
	\eqref{boundfinitesupport} $\Longrightarrow$ \eqref{bound} follows by Corollary~\ref{corollaryfinite}.\\
	\eqref{bound} $\Longrightarrow$ \eqref{ntqg} is proven by the argument given in the proof of \cite[Theorem 2.1]{O2015}, with only minor adjustments needed due to the fact that $\X$ has a quasi-norm instead of a norm. \\
	
	\eqref{boundfinite>m}$\Longrightarrow$ \eqref{boundfinitesupport}. Choose $m, \C_3$ as in \eqref{boundfinite>m}, and let $B:=\{1,\dots,m\}$.
	Note that for every set $D\subset \N$ with $0<|D|\le m$ and every $x\in \X$, by Remark~\ref{remarkeasypart} we have
	\begin{eqnarray}
	\|P_{D}(x)\|\le m\alpha^m c\|x\|. \label{easypart}
	\end{eqnarray}
	Fix $x\not=0$ and $A$ a $t$-greedy set for $x$ with $|A|\in \mathbf{n}$. If $|A|\le m$, \eqref{easypart} already gives an upper bound for the projection, so we may assume $|A|>m$. If $A\cap B=\emptyset$, then $A$ is a $t$-greedy set for $x-P_B(x)$. Thus, as $m<\supp{(x-P_B(x))}$, by hypothesis and \eqref{easypart} it follows that 
	\begin{eqnarray*}
		\|P_A(x)\|&=&\|P_A(x-P_B(x))\|\le \C_3\|x-P_B(x)\|\le \C_3\alpha \|x\|+\C_3\alpha\|P_B(x)\|\\
		&\le&  \C_3\alpha(1+m\alpha^mc)\|x\|. 
	\end{eqnarray*}
	On the other hand, if $A\cap B\not=\emptyset$, choose $D>A\cup B$ with $|D|=|A\cap B|$, and define 
	$$
	y:=x-P_B(x)+2c\|x\|\bff_{D}.
	$$
	For all $i\in D$ and all $j\not \in D$,
	\begin{equation}
	|\be_j^*(y)|=|\be_j^*(x-P_B(x))|\le |\be_j^*(x)\|\le c\|x\|\le |\be_i^*(y)|\le t^{-1}|\be_i^*(y)|. \label{notinD}
	\end{equation}
	For all $i\in A\setminus B$ and $j\not \in A \cup B\cup D$, 
	\begin{equation}
	|\be_j(y)|= |\be_j(x)|\le t^{-1}|\be_i(x)|=t^{-1}|\be_i(y)|.\label{notinBAminusBD}
	\end{equation}
	Finally, for $i\in A\setminus B$ and $j \in B$,
	\begin{equation}
	|\be_j(y)|=0\le t^{-1}|\be_i(y)|.\label{thirdone}
	\end{equation} 
	It follows from \eqref{notinD}, \eqref{notinBAminusBD} and \eqref{thirdone} that if $i\in (A\setminus B)\cup D$ and $j\not \in (A\setminus B)\cup D$, then
	$$
	|\be_j(y)|\le t^{-1}|\be_i(y)|.
	$$
	Thus, $(A\setminus B)\cup D$ is a $t$-greedy set for $y$. As $|(A\setminus B)\cup D|=|A|\in \mathbf{n}$ and $m<\supp{(y)}$, by hypothesis and \eqref{easypart} we have 
	\begin{eqnarray}
	\|P_{(A\setminus B)\cup D}(y)\|&\le& \C_3\|y\|=\C_3\big|\big|x-P_B(x)+2c\|x\|\bff_{D}\big|\big|\nonumber\\
	&\le& \C_3\alpha^2(\|x\|+\|P_B(x)\|+2c\|x\|\|\bff_{D}\|)\nonumber\\
	&\le& \C_3\alpha^2(1+m\alpha^{m}c+2c\alpha^{|D|-1}|D|c)\|x\|\nonumber\\
	&\le &4 m\C_3\alpha^{m+2}c^2\|x\|. \label{addingD}
	\end{eqnarray}
	From this and \eqref{addingD} we get
	\begin{eqnarray}
	\|P_{(A\setminus B)}(x)\|&=&\|P_{(A\setminus B)}(y)\| \le \alpha \|P_{(A\setminus B)\cup D}(y)\|+\alpha \|P_D(y)\|\nonumber\\
	&=& \alpha \|P_{(A\setminus B)\cup D}(y)\|+\alpha\big|\big|P_D(x)+2c\|x\|\bff_{D}\big|\big|\nonumber\\
	&\le& \alpha \|P_{(A\setminus B)\cup D}(y)\|+\alpha^2\|P_D(x)\|+2c\alpha^2\|\bff_D\|\|x\|\nonumber\\
	&\le&  \alpha \|P_{(A\setminus B)\cup D}(y)\|+\alpha^{2}m\alpha^mc\|x\|+2c\alpha^2\alpha^{|D|-1}|D|c\|x\|\nonumber\\
	&\le &7m\C_3\alpha^{m+3}c^2\|x\|.\nonumber%\label{part2}
	\end{eqnarray}
	Combining this result with \eqref{easypart} we obtain
	$$
	\|P_A(x)\|\le \alpha\|P_{A\cap B}(x)\|+\alpha\|P_{A\setminus B}(x)\|\le (m\alpha^{m+1}c+7m\C_3\alpha^{m+3}c^2)\|x\|.
	$$
	\eqref{pointwisebounded} $\Longrightarrow$ \eqref{boundfinite>m}. Suppose the implication is false. Then, there is $x_1$ with finite support and $A_1$ a $t$-greedy set for $x_1$ with $|A_1|\in\mathbf{n}$ such that
	$$
	\|x_1\|\le \frac{1}{10\alpha c} \qquad \text{and}\qquad \|P_{A_1}(x_1)\|>10\alpha.
	$$
	Let $m_0=0$ and $m_1:=\max\supp{(x_1)}+1$. As \eqref{boundfinite>m} does not hold, there is $x_2$ with finite support and a $t$-greedy set $A_2$ for $x_2$ with $|A_2|\in \n$ such that the following hold: 
	\begin{eqnarray}
	m_1&<&\supp{(x_2)},\nonumber\\%\label{awayfromx1}\\
	\|x_2\|&\le& \frac{\min_{j\in \supp{(x_1)}}\{|\be_j^*(x_1)|\}}{10^2m_1\alpha^{m_1+1}c}\nonumber\\%\label{firstconditions}\\
	\|P_{A_2}(x_2)\|&>&(10\alpha)^2(1+\alpha^{m_1}m_1c).\nonumber%\label{largeprojection}
	\end{eqnarray}
	By an inductive argument, we can choose a sequence of elements with finite support $(x_i)_i\subset \X$ and corresponding $t$-greedy sets $(A_i)$ with $|A_i|\in \mathbf{n}$, and a strictly increasing scalar sequence $(m_i)_{i}$ such that, for all $i$,  
	\begin{eqnarray}
	m_i&<&\supp{(x_{i+1})}< m_{i+1},\label{awayfromx1i}\\
	\|x_{i+1}\|&\le& \frac{\min_{1\le l\le i} \min_{j\in \supp{(x_l)}}\{|\be_j^*(x_l)|\}}{10^{i+1}m_i\alpha^{m_i+1}c}\label{firstconditionsi}\\
	\|P_{A_{i+1}}(x_{i+1})\|&>&\alpha^{m_i+1}m_ic+1.\label{largeprojectioni}
	\end{eqnarray}
	It follows from \eqref{firstconditionsi} that if $1\le l<i$, then
	\begin{equation}
	\max_{k\in \supp{(x_i)}}|\be_{k}^*(x_{i})|\le c \|x_{i+1}\|<\min_{k\in \supp{(x_l)}}|\be_{k}^*(x_{l})|.\label{smaller5}
	\end{equation}
	For each $n\in \N$, let 
	$$
	y_n:=\sum\limits_{i=1}^{n}x_i\qquad \text{and} \qquad B_{n}:=\supp{(y_n)}. 
	$$
	Note that for every $j\in \N$, $\|x_j\|\le 1$. Thus, for every $i\in \N$ and every $A\subset \supp{(x_{i+1})}$ such that $|A|\le m_i$, by Remark~\ref{remarkeasypart} we have
	\begin{equation}
	P_A(x_{i+1})\le \alpha^{|A|-1}|A|c\|x_{i+1}\|\le \alpha^{m_i}m_ic.\nonumber%\label{smallprojections}
	\end{equation}
	It follows from this, \eqref{awayfromx1i} and \eqref{largeprojectioni}  that 
	\begin{equation}
	|A_{i+1}|>m_i>|B_i|. \label{largeenough5}
	\end{equation}
	Thus, for every $i$, there is $C_{i+1}\subset A_{i+1}$ such that $|C_{i+1}|=|B_i|$ and, for each $j\in C_{i+1}$ and each $l\in A_{i+1}\setminus C_{i+1}$, 
	\begin{equation}
	|\be_j^*(x_{i+1})|\le |\be_l^*(x_{i+1})|.\label{stilltgreedy}
	\end{equation}
	Let
	$$
	D_{i+1}=B_i\cupdot (A_{i+1}\setminus C_{i+1}). 
	$$
	As $|C_{i+1}|=|B_i|$, it follows that $|D_{i+1}|=|A_{i+1}|\in \mathbf{n}$ for all $i$. Let 
	$$
	y:=\lim_{n\to +\infty}y_n. 
	$$
	By our choice of $x_1$ and \eqref{firstconditionsi}, inductively it follows that $|\be_j^*(x_i)|\le 1$ for all $i,j$. Hence, again by \eqref{firstconditionsi}, $y$ is well defined. We claim that for all $i$, $D_{i+1}$ is a $t$-greedy set for $y$. To prove this claim, fix $i\in \N$, $j\in D_{i+1}$, and $n\in \supp{(y)}\setminus D_{i+1}$. As 
	$$
	\supp{(y)}=\bigcup_{k}\supp{(x_k)}
	$$
	and these supports are pairwise disjoint, there are unique $k_0, k_1\in \N$ such that 
	$$
	j\in \supp{(x_{k_0})}\qquad \text{and} \qquad n\in \supp{(x_{k_1})}.
	$$
	Since $n\not \in D_{i+1}$, $j\in D_{i+1}$ and 
	$$
	\bigcup_{1\le k\le i}\supp{(x_k)}=B_i\subset D_{i+1}=B_i\cup (A_{i+1}\setminus C_{i+1})\subset B_{i}\cup \supp{(x_{i+1})},
	$$
	it follows that $k_0\le i+1\le k_1$. Now if $k_0<k_1$ it follows by \eqref{smaller5} that 
	$$
	|\be_{n}^*(y)|=|\be_n^*(x_{k_1})|<|\be_j^*(x_{k_0})|=|\be_j^*(y)|.
	$$
	On the other hand, if $k_0=k_1=i+1$, either $n\in C_{i+1}$ or $n\not \in A_{i+1}$. As $j\in A_{i+1}\setminus C_{i+1}$, in the first case by \eqref{stilltgreedy} we obtain 
	$$
	|\be_{n}^*(y)|=|\be_n^*(x_{i+1})|\le |\be_j^*(x_{i+1})|=|\be_j^*(y)|, 
	$$
	whereas in the second case, the fact that $A_{i+1}$ is a $t$-greedy set for $x_{i+1}$ entails that 
	$$
	|\be_{n}^*(y)|=|\be_n^*(x_{i+1})|\le t^{-1}|\be_j^*(x_{i+1})|=|\be_j^*(y)|. 
	$$
	Since all cases have been considered, our claim is proven. Therefore,  by hypothesis, there is $\C_4(y)>0$ such that 
	\begin{equation}
	\|P_{D_{i+1}}(y)\|\le \C_4(y),\label{boundforPDi+1}
	\end{equation}
	for all $i\in \N$. As $|C_{i+1}|=|B_i|< m_1$ by \eqref{largeenough5} and $C_{i+1}\subset A_{i+1}\subset\supp{(x_{i+1})}$, by \eqref{firstconditionsi} and Remark~\ref{remarkeasypart} we have
	\begin{equation}
	\|P_{C_{i+1}}(y)\|=\|P_{C_{i+1}}(x_{i+1})\|\le \alpha^{|C_{i+1}|-1}|C_{i+1}|c\|x_{i+1}\|\le m_1\alpha^{m_1}c\|x_{i+1}\|\le 1. \label{boundforCi+1}
	\end{equation}
	Also, for all $i$, 
	\begin{equation}
	\|P_{B_{i}}(y)\|=\|y_i\|\le \sum\limits_{j=1}^{i-1}\alpha^{j} \|x_j\|+\alpha^{i-1}\|x_i\|\le \sum\limits_{j=1}^{i-1}\frac{1}{10^j}+\frac{1}{10^i}\le 1.\label{anotherone}
	\end{equation}
	Since $P_{A_{i+1}}(y)=P_{A_{i+1}}(x_{i+1})$ and $c> 2$, from \eqref{largeprojectioni}, \eqref{boundforPDi+1}, \eqref{boundforCi+1} and \eqref{anotherone} it follows that 
	\begin{eqnarray}
	\C_4(y)&\ge& \|P_{D_{i+1}}(y)\|=\|P_{B_i}(y)+P_{A_{i+1}}(y)-P_{C_{i+1}}(y)\|\nonumber\\
	&\ge& \alpha^{-1}\|P_{A_{i+1}}(y)\|-\|P_{B_{i}}(y)-P_{C_{i+1}}(y)\|\nonumber\\
	&\ge&\alpha^{-1}\|P_{A_{i+1}}(y)\|-\alpha\|P_{B_i}(y)\|-\alpha\|P_{C_{i+1}}(y)\|\nonumber\\
	&\ge& \alpha^{-1}(\alpha^{m_i+1}m_ic+1)-2\alpha\ge 2\alpha(m_i-1)\nonumber
	\end{eqnarray}
	for all $i\in \N$, which is impossible.
\end{proof}
\begin{remark}\rm In the particular case of $t=1$ and $\mathbf{n}=\N$,  Theorem~\ref{theoremqbextension} gives an alternative proof of a known result for quasi-greedy bases in quasi-Banach spaces. 
\end{remark}

\section{Open question}\label{section8}

\begin{question} Theorem \ref{theoremboundedgapsv2} is proven under the condition of Schauder bases. Can the Schauder condition be removed and the result be obtained using only the Markushevich condition?
\end{question}

\end{document}